\documentclass[11pt,reqno]{article}
\usepackage{mathrsfs}
\usepackage{amsmath,amssymb,amsthm}
\usepackage{hyperref}
\usepackage[english]{babel}
\newtheorem{theorem}{Theorem}[section]
\newtheorem{lemma}[theorem]{Lemma}
\newtheorem{proposition}[theorem]{Proposition}

\newtheorem{definition}[theorem]{Definition}
\newtheorem{remark}[theorem]{Remark}
\numberwithin{equation}{section}

\title{On stability and instability of the ground states for the focusing inhomogeneous NLS with inverse-square potential}
\author{{\bf JinMyong An, HakBom Mun and JinMyong Kim$^*$}\\
\footnotesize{Faculty of Mathematics, {\bf Kim Il Sung} University, Pyongyang, Democratic People's Republic of Korea}\\
\footnotesize{$^*$ Corresponding Author: jm.kim0211@ryongnamsan.edu.kp.}
}
\date{}
\begin{document}
\maketitle
\begin{abstract}
In this paper, we study the stability and instability of the ground states for the focusing inhomogeneous nonlinear Schr\"{o}dinger equation with inverse-square potential (for short, INLS$_c$ equation):
\[iu_{t} +\Delta u+c|x|^{-2}u+|x|^{-b} |u|^{\sigma } u=0,\;
u(0)=u_{0}(x) \in H^{1},\;(t,x)\in \mathbb R\times\mathbb R^{d},\]
where $d\ge3$, $0<b<2$, $0<\sigma<\frac{4-2b}{d-2}$ and $c\neq 0$ be such that $c<c(d):=\left(\frac{d-2}{2}\right)^{2}$.
In the mass-subcritical case $0<\sigma<\frac{4-2b}{d}$, we prove the stability of the set of ground states for the INLS$_{c}$ equation.
In the mass-critical case $\sigma=\frac{4-2b}{d}$, we first prove that the solution of the INLS$_c$ equation with initial data $u_{0}$ satisfying $E(u_0)<0$ blows up in finite or infinite time. Using this fact, we then prove that the ground state standing waves are unstable by blow-up.
In the intercritical case $\frac{4-2b}{d}<\sigma<\frac{4-2b}{d-2}$, we finally show the instability of ground state standing waves for the INLS$_c$ equation.
\end{abstract}


\section{Introduction}\label{sec 1.}

In this paper, we consider the Cauchy problem for the focusing inhomogeneous nonlinear
Schr\"{o}dinger equation with inverse-square potential, denoted by INLS$_{c}$ equation,
\begin{equation} \label{GrindEQ__1_1_}
\left\{\begin{array}{l} {iu_{t}+\Delta u+c|x|^{-2}u+|x|^{-b}|u|^{\sigma } u=0,\;(t,x)\in\mathbb R\times\mathbb R^{d},}
\\ {u\left(0,\; x\right)=u_{0}(x)\in H^{1}(\mathbb R^{d}),} \end{array}\right.
\end{equation}
where $d\ge3$, $u:\mathbb R\times \mathbb R^{d} \to \mathbb C$, $u_{0}:\mathbb R^{d} \to \mathbb C$, $0<b<2$, $0<\sigma<\frac{4-2b}{d-2}$ and $c\neq 0$ satisfies $c<c(d):=\left(\frac{d-2}{2}\right)^{2}$.

The equation \eqref{GrindEQ__1_1_} appears in a variety of physical settings, for example, in nonlinear optical systems with spatially dependent interactions (see e.g. \cite{BPVT07,KMVBT17} and the references therein).
The case $b=c=0$ is the classic nonlinear Schr\"{o}dinger (NLS) equation which has been been widely studied over the last three decades (see e.g. \cite{C03, LP15, WHHG11} and the references therein). The case $b=0$ and $c\neq0$ is known as the NLS equation with inverse-square potential, denoted by NLS$_{c}$ equation, has also attracted a lot of interest in recent years (see e.g. \cite{BDZ18,D18,D21,KMVZZ17,KMVZ17,MM18, Y21} and the references therein). Moreover, when $c=0$ and $b\neq0$, we have the inhomogeneous nonlinear Schr\"{o}dinger equation, denoted by INLS equation, which has been extensively studied in the past several years (see e.g. \cite{AT21,AK211,AK212,AK23,AKC22,AC21,BF15,C21,DK21,G12} and the references therein).
When $b\neq 0$ and $c\neq 0$, we have the INLS$_{c}$ equation, which has also been studied by several authors in recent years (see \cite{AJK23,CG21,JAK21,S16} for example).

In this paper, we focus on the INLS$_{c}$ equation \eqref{GrindEQ__1_1_} with $b>0$ and $c\neq 0$.
The INLS$_c$ equation \eqref{GrindEQ__1_1_} is invariant under the scaling,
$$
u_{\lambda}(t,x):=\lambda^{\frac{2-b}{\sigma}}u\left(\lambda^{2}t,\lambda x\right),~\lambda>0.
$$
An easy computation shows that
$$
\left\|u_{\lambda}(0)\right\|_{\dot{H}^{s}}=\lambda^{s-\frac{d}{2}+\frac{2-b}{\sigma}}\left\|u_0\right\|_{\dot{H}^{s}},
$$
which implies that the critical Sobolev index is given by
\begin{equation}\nonumber
s_{c}:=\frac{d}{2}-\frac{2-b}{\sigma}.
\end{equation}
If $s_{c}=0$ (alternatively $\sigma=\frac{4-2b}{d}$) the problem is known as the mass-critical or $L^{2}$-critical. If $s_{c}=1$ (alternatively $\sigma=\frac{4-2b}{d-2}$) it is called energy-critical or $H^1$-critical. The problem is known as intercritical (mass-supercritical and energy-subcritical) if $0<s_{c}<1$ (alternatively $\frac{4-2b}{d}<\sigma<\frac{4-2b}{d-2}$).
On the other hand, solutions to the INLS$_{c}$ equation \eqref{GrindEQ__1_1_} conserve the mass and energy, defined respectively by
\begin{equation}\label{GrindEQ__1_2_}
M\left(u(t)\right):=\int_{\mathbb R^{d}}{\left|u(t, x)\right|^{2}dx},
\end{equation}
\begin{equation}\label{GrindEQ__1_3_}
E\left(u(t)\right):=\frac{1}{2}\left\|u\right\|_{\dot{H}_{c}^{1}}^{2}-\frac{1}{\sigma+2}\int_{\mathbb R^{d}}{|x|^{-b} \left|u(t, x)\right|^{\sigma +2} dx},
\end{equation}
where
\begin{equation}\nonumber
\left\|u\right\|_{\dot{H}_{c}^{1}}^{2}:=\left\|\nabla u \right\|_{L^2}^{2}-c\left\||x|^{-1}u\right\|_{L^2}^{2}.
\end{equation}
is the Hardy functional. By sharp Hardy inequality
\begin{equation}\label{GrindEQ__1_4_}
c(d)\left\||x|^{-1}u\right\|_{L^2}^2\le \left\|\nabla u\right\|_{L^2}^{2},~\forall u\in H^1(\mathbb R^d),
\end{equation}
we can see that
$$
\left\|u\right\|_{\dot{H}_{c}^1}\sim\left\|u\right\|_{\dot{H}^1},~\textrm{for}~c<c(d).
$$

Let us recall the known results for INLS$_{c}$ equation \eqref{GrindEQ__1_1_} with $b>0$ and $c\neq 0$.
Using the energy method, Suzuki \cite{S16} showed that if
\footnote[1]{\ Note that the author in \cite{S16} considered \eqref{GrindEQ__1_1_} with $c=c(d)$. The authors in \cite{CG21} pointed out that the proof for the case $c>-c(d)$ is an immediate consequence of the previous one.}
$d\ge3$, $0<\sigma<\frac{4-2b}{d-2}$, $c<c(d)$ and $0<b<2$, then the INLS$_{c}$ equation \eqref{GrindEQ__1_1_} is locally well-posed in $H^{1}$. It was also proved that any local solution of the INLS$_c$ equation \eqref{GrindEQ__1_1_} extends globally in time if $0<\sigma<\frac{4-2b}{d}$.
Later, the authors in \cite{AJK23,CG21} studied the global existence and blow-up of $H^{1}$-solutions to the focusing INLS$_{c}$ equation \eqref{GrindEQ__1_1_} with $d\ge 3$ and $\frac{4-2b}{d}\le\sigma<\frac{4-2b}{d-2}$.
The authors in \cite{JAK21} also established the local well-posedness as well as small data global well-posedness and scattering in $H^{1}$ for the energy-critical INLS$_{c}$ equation \eqref{GrindEQ__1_1_} with $d\ge3$, $c<\frac{(d+2-2b)^{2}-4}{(d+2-2b)^{2}}c(d)$, $0<b<\frac{4}{d}$ and $\sigma=\frac{4-2b}{d-2}$.

On the other hand, the stability and instability of the ground states for \eqref{GrindEQ__1_1_} have also been studied by several authors.
The stability and instability of the ground state standing waves for the classical nonlinear Schr\"{o}dinger (i.e. \eqref{GrindEQ__1_1_} with $b=0$ and $c=0$) were widely studied by physicists and mathematicians (see. e.g. \cite{F15}). The case $b=0$ and $c\neq0$ was studied by \cite{BDZ18,D21} and the case $0<b<2$ and $c=0$ was studied by \cite{AC21,BF15,G12}. However, up to the knowledge of the authors, there are no any results about the stability and instability of the ground states for the INLS$_{c}$ equation \eqref{GrindEQ__1_1_} with $0<b<2$ and $c\neq0$.

The main purpose of this paper is to study the stability and instability of the ground states for the INLS$_c$ equation \eqref{GrindEQ__1_1_} with $0<b<2$ and $c\neq0$.

Throughout the paper, we call a standing wave a solution of \eqref{GrindEQ__1_1_} of the form $e^{i\omega t}\phi_{\omega}$, where $\omega\in \mathbb R$ is a frequency and $\phi_{\omega}\in H^1$ is a nontrivial solution to the elliptic equation
\begin{equation}\label{GrindEQ__1_5_}
-\Delta \phi_{\omega}+\omega \phi_{\omega}-c|x|^{-2}\phi_{\omega}-|x|^{-b}|\phi_{\omega}|^{\sigma}\phi_{\omega}=0.
\end{equation}

Note that \eqref{GrindEQ__1_5_} can be written as $S_{\omega}'(u)=0$, where $S_{\omega}(u)$ is the action functional defined by
\begin{equation}\label{GrindEQ__1_6_}
S_{\omega}(u):=E(u)+\frac{\omega}{2}\left\|u\right\|_{L^2}^{2}.
\end{equation}

We denote the set of non-trivial solutions of \eqref{GrindEQ__1_5_} by
\begin{equation}\label{GrindEQ__1_7_}
\mathcal{A}_{\omega}:=\{u\in H^{1}\setminus \{0\}:~S_{\omega}'(u)=0\}.
\end{equation}

\begin{definition}[Ground states]\label{defn 1.1.}
A function $\phi\in \mathcal{A}$ is called a ground state for \eqref{GrindEQ__1_5_} if it is a minimizer of $S_{\omega}$ over the set $\mathcal{A}_{\omega}$. The set of ground states is denoted by $\mathcal{G}_{\omega}$. In particular,
\begin{equation}\label{GrindEQ__1_8_}
\mathcal{G}_{\omega}=\{u\in \mathcal{A}_{\omega},~S_{\omega}(u)\le S_{\omega}(v),~\forall v\in \mathcal{A}_{\omega}\}.
\end{equation}
\end{definition}

We have the following result on the existence of ground states for \eqref{GrindEQ__1_7_}, which is proved in Section \ref{sec 2.}.
\begin{proposition}\label{prp 1.2.}
Let $d\ge 3$, $0<b<2$, $0<\sigma<\frac{4-2b}{d-2}$, $\omega>0$ and $c\neq 0$ be such that $c<c(d)$. The the set of ground states $\mathcal{G}_{\omega}$ is not empty and it is characterized by
\begin{equation}\label{GrindEQ__1_9_}
\mathcal{G}_{\omega}=\{u\in H^1:~S_{\omega}(u)=s(\omega),~K_{\omega}(u)=0\},
\end{equation}
where
\begin{equation}\label{GrindEQ__1_10_}
K_{\omega}(u):=\partial_{\lambda}S_{\omega}(\lambda u)|_{\lambda=1}=\left\|u\right\|_{\dot{H}_{c}^{1}}^{2}+\omega \left\|u\right\|_{L^2}^{2}-\left\||x|^{\frac{-b}{\sigma+2}}u\right\|_{L^{\sigma+2}}^{\sigma+2}
\end{equation}
is the Nehari functional and
\begin{equation}\label{GrindEQ__1_11_}
s(\omega):=\inf\{S_{\omega}(u):~u\in H^{1}\setminus\{0\},~K_{\omega}(u)=0\}.
\end{equation}
\end{proposition}

Using Proposition \ref{prp 1.2.}, we first show the following stability of the set of ground states for the INLS$_{c}$ equation \eqref{GrindEQ__1_1_} in the $L^2$-subcritical case $0<\sigma<\frac{4-2b}{d}$.

\begin{theorem}\label{thm 1.3.}
Let $d\ge 3$, $0<b<2$, $0<\sigma<\frac{4-2b}{2}$, $\omega>0$ and $c\neq 0$ be such that $c<c(d)$.
Then the set of ground states $\mathcal{G}_{\omega}$ is stable in the following sense: for any $\epsilon>0$ there exists $\delta>0$ such that for any initial data $u_{0}$ satisfying
$$
\inf_{\phi_{\omega}\in\mathcal{G}_{\omega}}\left\|u_{0}-\phi_{\omega}\right\|_{H^1}<\delta
$$
the corresponding solution $u(t)$ to \eqref{GrindEQ__1_1_} with initial data $u_{0}$ satisfies
$$
\inf_{\phi_{\omega}\in \mathcal{G}_{\omega}}\left\|u(t)-\phi_{\omega}\right\|_{H^1}<\epsilon,
$$
for all $t\in \mathbb R$.
\end{theorem}

\begin{remark}\label{rem 1.4.}
\textnormal{Theorem \ref{thm 1.3.} shows that if $d(u_{0,n},\mathcal{G}_{\omega})\to 0$ as $n\to \infty$, then the corresponding solution $u_{n}(t)$ of \eqref{GrindEQ__1_1_} with initial data $u_{0,n}\in H^1$ satisfies
$$
d(u_{n}(t),\mathcal{G}_{\omega})\to 0,~\forall t\in \mathbb R,
$$
where
$$
d(f,X):=\inf_{g \in X}\left\|f-g\right\|_{H^1}~\textrm{for}~f\in H^1~\textrm{and}~X\subset H^1.
$$
}\end{remark}

Next, we study the instability of the ground state standing waves for \eqref{GrindEQ__1_1_} in the mass-critical case $\sigma=\frac{4-2b}{d}$. To this end, we need the following blow-up result for \eqref{GrindEQ__1_1_}.

\begin{proposition}\label{prp 1.5.}
Let $d\ge 3$, $0<b<2$, $\sigma=\frac{4-2b}{d}$ and $c\neq 0$ be such that $c<c(d)$. Let $u$ be the solution to \eqref{GrindEQ__1_1_} defined on the maximal forward time interval of existence $[0,T^{*})$. If $E(u_0)<0$, then either $T^{*}<\infty$ or $T^{*}=\infty$ and  there exits a time sequence $t_n\to \infty$ such that $\left\|u(t_n)\right\|_{\dot{H}^1}\to\infty$ as $n\to \infty$.
A similar statement holds for negative times.
\end{proposition}

\begin{remark}\label{rem 1.6.}
\textnormal{Proposition \ref{prp 1.5.} extends the Theorem 1.2 of \cite{CG21} for radial or finite-variance data to non-radial data.}
\end{remark}
\begin{theorem}\label{thm 1.7.}
Let $d\ge 3$, $0<b<2$, $\sigma=\frac{4-2b}{d}$, $\omega>0$ and $c\neq 0$ be such that $c<c(d)$.
If $\phi_{\omega}\in \mathcal{G}_{\omega}$, then the ground state standing wave $e^{i\omega t}\phi_{\omega}(x)$ of \eqref{GrindEQ__1_1_} is unstable in $H^1$ in the following sense: for any $\epsilon>0$, there exists $u_{0}\in H^1$ such that $\left\|u_{0}-\phi_{\omega}\right\|_{H^1}<\epsilon$ and the solution $u(t)$ of \eqref{GrindEQ__1_1_} with initial data $u_{0}$ blows up in finite or infinite time.
\end{theorem}

\begin{remark}\label{rem 1.8.}
\textnormal{Theorem \ref{thm 1.7.} shows that there exists $\{u_{0,n}\}\subset H^1$ such that $u_{0,n}\to \phi_{\omega}$ in $H^1$ as $n\to \infty$ and the corresponding solution $u_n$ of \eqref{GrindEQ__1_1_} with initial data $u_{0,n}$ blows up in finite or infinite time for any $n\ge 1$.}
\end{remark}

\begin{remark}\label{rem 1.9.}
\textnormal{Theorem \ref{thm 1.7.} extends the instability results of \cite{BDZ18,G12} for the mass-critical NLS$_{c}$ equation and INLS equation to the INLS$_{c}$ equation.}
\end{remark}

Finally, we have the following instability result in the intercritical case $\frac{4-2b}{d}<\sigma<\frac{4-2b}{d-2}$.
\begin{theorem}\label{thm 1.10.}
Let $d\ge 3$, $0<b<2$, $\frac{4-2b}{d}<\sigma<\frac{4-2b}{d-2}$, $\omega>0$ and $c\neq 0$ be such that $c<c(d)$.
If $\phi_{\omega}\in \mathcal{G}_{\omega}$, then the ground state standing wave $e^{i\omega t}\phi_{\omega}(x)$ of \eqref{GrindEQ__1_1_} is unstable in $H^1$ in the sense of Theorem \ref{thm 1.7.}
\end{theorem}
\begin{remark}\label{rem 1.11.}
\textnormal{Theorem \ref{thm 1.10.} extends the instability results of \cite{AC21,D21} for the intercritical INLS equation and NLS$_{c}$ equation to the INLS$_{c}$ equation.}
\end{remark}

This paper is organized as follows. In Section 2, we prove the existence of ground states (Proposition \ref{prp 1.2.}). In Section 3, we study the stability and instability of the ground states for the INLS$_c$ equation \eqref{GrindEQ__1_1_}.

\section{Existence and variational characterization of ground states}\label{sec 2.}

In this section, we prove Proposition \ref{prp 1.2.}.
Let us denote the $\omega$-Hardy functional by
$$
H_{\omega}(u):=\left\|u\right\|_{\dot{H}_{c}^{1}}^{2}+\omega \left\|u\right\|_{L^2}^{2}.
$$
Using the sharp Hardy inequality \eqref{GrindEQ__1_4_}, we can see that for $c<c(d)$ and $\omega>0$ fixed,
\begin{equation}\label{GrindEQ__2_1_}
H_{\omega}(u)\sim\left\|u\right\|_{H^1}^{2}.
\end{equation}
We note that the action functional can be rewritten as
\begin{equation}\label{GrindEQ__2_2_}
S_{\omega}(u)=\frac{1}{2}K_{\omega}(u)+\frac{\sigma}{2(\sigma+2)}\left\||x|^{\frac{-b}{\sigma+2}}u\right\|_{L^{\sigma+2}}^{\sigma+2}
=\frac{1}{\sigma+2}K_{\omega}(u)+\frac{\sigma}{2(\sigma+2)}H_{\omega}(u).
\end{equation}

We recall the following sharp Gagliardo--Nirenberg inequality.
\begin{lemma}[\cite{CG21}]\label{lem 2.1.}
Let $d\ge 3$, $0<\sigma<\frac{4-2b}{d-2}$, $c>-c(d)$ and $0<b<2$. Then for $f\in H^{1}$, we have
\begin{equation}\label{GrindEQ__2_3_}
\left\||x|^{\frac{-b}{\sigma+2}}f\right\|_{L^{\sigma+2}}^{\sigma+2}\le C_{GN}\left\|f\right\|_{\dot{H}_{c}^{1}}^{\frac{d\sigma+2b}{2}}
\left\|f\right\|_{L^{2}}^{\frac{4-2b-\sigma(d-2)}{2}}.
\end{equation}
The equality in \eqref{GrindEQ__2_3_} is attained by a function $\phi_{1}\in H^{1}$, which is a positive solution of \eqref{GrindEQ__1_5_} with $\omega=1$.
\end{lemma}
\begin{lemma}[\cite{CG16,CG21}]\label{lem 2.2.}
If $d\ge 3$, $0<b<2$ and $0<\sigma<\frac{4-2b}{d-2}$, then $H^{1}$ is compactly embedded in $L^{\sigma+2}\left(|x|^{-b}dx\right)$.
\end{lemma}
\begin{lemma}\label{lem 2.3.}
Let $\omega>0$. Then we have $s(\omega)>0$, where $s(\omega)$ is given in \eqref{GrindEQ__1_11_}.
\end{lemma}
\begin{proof}
Let $u\in H^1(\mathbb R^d)\setminus \{0\}$ be such that $K_{\omega}(u)=0$. By \eqref{GrindEQ__2_2_}, we can see that
$$
\left\||x|^{\frac{-b}{\sigma+2}}u\right\|_{L^{\sigma+2}}^{\sigma+2}=H_{\omega}(u).
$$
Using this fact and Lemma \ref{lem 2.1.}, we infer that
$$
\left\||x|^{\frac{-b}{\sigma+2}}u\right\|_{L^{\sigma+2}}^{\sigma+2} \le C H_{\omega}(u)^{\frac{\sigma+2}{2}}=
C \left(\left\||x|^{\frac{-b}{\sigma+2}}u\right\|_{L^{\sigma+2}}^{\sigma+2}\right)^{\frac{\sigma+2}{2}},
$$
for some $C>0$. This yields that
$$
\left\||x|^{\frac{-b}{\sigma+2}}u\right\|_{L^{\sigma+2}}^{\sigma+2}\ge C^{-\frac{2}{\sigma}}>0.
$$
Hence, we can see that
$$
S_{\omega}(u)=\frac{\sigma}{2(\sigma+2)}\left\||x|^{\frac{-b}{\sigma+2}}u\right\|_{L^{\sigma+2}}^{\sigma+2}\ge \frac{\sigma}{2(\sigma+2)}C^{-\frac{2}{\sigma}}.
$$
Taking the infimum over $u\in H^1(\mathbb R^d)\setminus \{0\}$ with $K_{\omega}(u)=0$, we get the desired result.
\end{proof}

We denote the set of all minimizers of \eqref{GrindEQ__1_11_} by
$$
\mathcal{M}_{\omega}:=\{u\in H^1:~S_{\omega}(u)=s(\omega),~K_{\omega}(u)=0\}
$$
\begin{lemma}\label{lem 2.4.}
Let $\omega>0$. The set $\mathcal{M}_{\omega}$ is not empty.
\end{lemma}
\begin{proof}
Let $\{u_n\}$ be a minimizing sequence of $s(\omega)$, i.e. $u_{n}\in H^1(\mathbb R^d)\setminus \{0\}$, $K_{\omega}(u_n)=0$ and $S_{\omega}(u_n)\to s(\omega)$ as $n\to \infty$. Using the fact that $K_{\omega}(u_n)=0$ and \eqref{GrindEQ__2_2_}, we have
\begin{equation}\label{GrindEQ__2_4_}
\frac{\sigma}{2(\sigma+2)}H_{\omega}(u_n)=\frac{\sigma}{2(\sigma+2)}\left\||x|^{\frac{-b}{\sigma+2}}u_n\right\|_{L^{\sigma+2}}^{\sigma+2}\to s(\omega),
\end{equation}
as $n\to \infty$. This shows that $H_{\omega}(u_n)$ is bounded.
Hence, in view of \eqref{GrindEQ__2_1_}, we can see that $\{u_n\}$ is bounded in $H^1$.
Therefore, there exists $\phi\in H^1$ such that, up to a subsequence, $u_{n}\rightharpoonup \phi$ weakly in $H^1$.
By Lemma \ref{lem 2.2.}, we see that as $n\to \infty$,
\begin{equation}\label{GrindEQ__2_5_}
\left\||x|^{\frac{-b}{\sigma+2}}u_n\right\|_{L^{\sigma+2}}^{\sigma+2}\to\left\||x|^{\frac{-b}{\sigma+2}}\phi\right\|_{L^{\sigma+2}}^{\sigma+2},
\end{equation}
which implies
\begin{eqnarray}\begin{split}\label{GrindEQ__2_6_}
K_{\omega}(\phi)&=H_{\omega}(\phi)-\left\||x|^{\frac{-b}{\sigma+2}}\phi\right\|_{L^{\sigma+2}}^{\sigma+2}\\
&\le \liminf_{n\to \infty}H_{\omega}(u_{n})-\lim_{n\to\infty}{\left\||x|^{\frac{-b}{\sigma+2}}u_n\right\|_{L^{\sigma+2}}^{\sigma+2}}\\
&\le \liminf_{n\to \infty} K_{\omega}(u_{n})=0
\end{split}\end{eqnarray}
\eqref{GrindEQ__2_4_} and \eqref{GrindEQ__2_5_} show that $\phi\neq 0$. Next, we suppose that $K_{\omega}(\phi)<0$. Putting
$$
\lambda:=\left(H_{\omega}(\phi)\right)^{\frac{1}{\sigma}}\left\||x|^{\frac{-b}{\sigma+2}}\phi\right\|_{L^{\sigma+2}}^{-\frac{\sigma+2}{\sigma}},
$$
we can see that $\lambda\in (0,1)$ and $K_{\omega}(\lambda \phi)=0$. Hence, it follows from \eqref{GrindEQ__1_11_}, \eqref{GrindEQ__2_4_} and \eqref{GrindEQ__2_5_} that
$$
s(\omega)\le S_{\omega}(\lambda\phi)=\lambda^{\sigma+2}\frac{\sigma}{2(\sigma+2)}\left\||x|^{\frac{-b}{\sigma+2}}\phi\right\|_{L^{\sigma+2}}^{\sigma+2}<s(\omega),
$$
which is a contradiction. Thus, we can see that $K_{\omega}(\phi)=0$. Using \eqref{GrindEQ__2_2_}, \eqref{GrindEQ__2_4_} and \eqref{GrindEQ__2_5_}, we also have $S_{\omega}(\phi)=s(\omega)$.
By definition of $\mathcal{M}_{\omega}$ (see \eqref{GrindEQ__1_9_}), we can deduce that $\phi\in \mathcal{M}_{\omega}$.
\end{proof}

\begin{lemma}\label{lem 2.5.}
Let $\omega>0$. Then $\mathcal{G}_{\omega}=\mathcal{M}_{\omega}$.
\end{lemma}
\begin{proof}
By Lemma \ref{lem 2.3.}, $\mathcal{M}_{\omega}$ is non-empty. Let $u\in \mathcal{M}_{\omega}$.
Then there exists a Lagrange multiplier $\lambda\in \mathbb R$ such that $S_{\omega}'(u)=\lambda K_{\omega}'(u)$. Hence, we have
\begin{equation}\nonumber
0=K_{\omega}(u)=\langle S_{\omega}'(u),u\rangle=\lambda\langle K_{\omega}'(u),u\rangle.
\end{equation}
One can easily see that
$$
K'_{\omega}(u)=-2\Delta u+2\omega u-2c|x|^{-2}u-(\sigma+2)|x|^{-b}|u|^{\sigma}u,
$$
which implies that
\begin{equation}\nonumber
\lambda\langle K_{\omega}'(u),u\rangle=2H_{\omega}(u)-(\sigma+2)\left\||x|^{\frac{-b}{\sigma+2}}u\right\|_{L^{\sigma+2}}^{\sigma+2}=-\sigma\left\||x|^{\frac{-b}{\sigma+2}}u\right\|_{L^{\sigma+2}}^{\sigma+2}<0.
\end{equation}
This shows that $\lambda=0$, i.e. $S_{\omega}'(u)=0$. Hence $u\in \mathcal{A}_{\omega}$. To prove $u\in \mathcal{G}_{\omega}$, it remains to prove that $S_{\omega}(u)\le S_{\omega}(v)$ for all $v\in \mathcal{A}_{\omega}$. If $v\in \mathcal{A}_{\omega}$, then we have $K_{\omega}(v)=\langle S_{\omega}'(v),v\rangle=0$. By definition of $\mathcal{M}_{\omega}$ we see that $S_{\omega}(u)\le S_{\omega}(v)$. Therefore, $\mathcal{M}_{\omega}\subset \mathcal{G}_{\omega}$.

Conversely, let $u\in \mathcal{G}_{\omega}$. Since $u\in \mathcal{A}_{\omega}$, we can see that
\begin{equation}\label{GrindEQ__2_7_}
K_{\omega}(u)=\langle S_{\omega}'(u),u\rangle=0.
\end{equation}
On the other hand, since $\mathcal{M}_{\omega}$ is not empty, we take $v \in \mathcal{M}_{\omega}$. Noticing $\mathcal{M}_{\omega}\subset \mathcal{G}_{\omega}$, it follows from the definitions of $\mathcal{G}_{\omega}$ and $\mathcal{M}_{\omega}$ that
\begin{equation}\label{GrindEQ__2_8_}
S_{\omega}(u)=S_{\omega}(v)=s(\omega).
\end{equation}
In view of \eqref{GrindEQ__2_6_} and \eqref{GrindEQ__2_7_}, we have $u\in \mathcal{M}_{\omega}$.
Hence, $\mathcal{G}_{\omega}\subset \mathcal{M}_{\omega}$.
This completes the proof.
\end{proof}

\begin{proof}[Proof of Proposition \ref{prp 1.2.}]
Proposition \ref{prp 1.2.} follows directly from Lemmas \ref{lem 2.4.} and \ref{lem 2.5.}.
\end{proof}

\section{Stability and instability of the ground states}
In this section, we study the stability/instability of the ground states for the INLS$_{c}$ equation \eqref{GrindEQ__1_1_}.

First, we recall the localized virial estimates.
To this end, we introduce a function $\theta:[0,\infty)\to [0,\infty)$ satisfying
\begin{equation}\nonumber
\theta (r)=\left\{\begin{array}{l} {r^2,~\textrm{if}\;0\le r\le 1,}
\\ {\textrm{0},~\textrm{if}\;r\ge 2,} \end{array}\right.
\textrm{and}~~\theta''(r)\le 2 ~~\textrm{for}~~r\ge 0.
\end{equation}
For $R>1$, we define the radial function $\varphi_{R}:\mathbb R^{d}\to [0,\;\infty)$:
\begin{equation}\label{GrindEQ__3_1_}
\varphi_{R}(x)=\varphi_{R}(r):=R^{2}\theta(r/R),\;r=|x|.
\end{equation}

Given a real valued function $a$, we also define the virial potential by
\begin{equation}\nonumber
V_{a}(t):=\int_{\mathbb R^d}{a(x)\left|u(t,x)\right|^{2}dx.}
\end{equation}

\begin{lemma}[Localized virial estimates, \cite{AJK23}]\label{lem 3.1.}
Let $d \ge 3$, $0<b<2$, $c>-c(d)$, $R>1$ and $\varphi_{R}$ be as in \eqref{GrindEQ__3_1_}. Let $u:I\times\mathbb R^{d}\to \mathbb C$ be a solution to the focusing INLS$_{c}$ equation \eqref{GrindEQ__1_1_}. Then for any $t\in I$,
\begin{equation} \nonumber
\frac{d^2}{dt^2}V_{\varphi_{R}}(t)\le 8 G(u(t))+CR^{-2}+CR^{-b}\left\|u(t)\right\|_{H_{c}^{1}}^{\sigma+2},
\end{equation}
where $\left\|u\right\|_{H_{c}^{1}}=\left\|u\right\|_{\dot{H}_{c}^{1}}+\left\|u\right\|_{L^2}$ and
\begin{equation}\label{GrindEQ__3_2_}
G(u):=\left\|u\right\|_{\dot{H}_{c}^{1}}^{2}-\frac{d\sigma+2b}{2(\sigma+2)}\left\||x|^{\frac{-b}{\sigma+2}}u\right\|_{L^{\sigma+2}}^{\sigma+2}.
 \end{equation}
\end{lemma}

In the following lemma, we obtain the Pohazaev-type identities which are satisfied by any solution to the elliptic equation \eqref{GrindEQ__1_5_}.
The proof follows multiplying \eqref{GrindEQ__1_5_} by $\phi_{\omega}$ and $x\cdot\phi_{\omega}$ and using integration by parts and we omit the details.

\begin{lemma}\label{lem 3.2.}
If $\phi_{\omega}\in H^1$ is a solution to \eqref{GrindEQ__1_5_}, then we have
\begin{equation}\label{GrindEQ__3_3_}
\omega\left\|\phi_{\omega}\right\|_{L^2}^{2}
=\frac{4-2b-(d-2)\sigma}{2(\sigma+2)}\left\||x|^{\frac{-b}{\sigma+2}}\phi_{\omega}\right\|_{L^{\sigma+2}}^{\sigma+2}
=\frac{4-2b+(d-2)\sigma}{d\sigma+2b}\left\|\phi_{\omega}\right\|_{\dot{H}_{c}^1}^{2}.
\end{equation}
\end{lemma}
We also recall the so called Brezis-Lieb's lemma.
\begin{lemma}[\cite{BL83}]\label{lem 3.3.}
Let $0<p<\infty$. Suppose that $f_{n}\to f$ almost everywhere and $\{f_{n}\}$ is a bounded sequence in $L^p$, then
$$
\lim_{n\to\infty}(\left\|f_{n}\right\|_{L^p}^p-\left\|f_{n}-f\right\|_{L^p}^p)=\left\|f\right\|_{L^p}^p.
$$
\end{lemma}
\subsection{Mass-subcritical case.}
In this subsection, we study stability of ground state for \eqref{GrindEQ__1_1_} in the $L^{2}$-subcritical case, i.e. we prove Theorem \ref{thm 1.3.}.

\begin{proof}[\textnormal{\textbf{Proof of Theorem \ref{thm 1.3.}}}]
It was proved in \cite{S16} that the solution $u$ of \eqref{GrindEQ__1_1_} with $\sigma<\frac{4-2b}{d}$ exists globally. Assume by contradiction that there exists $\epsilon_{0}$ and a sequence $\{u_{0,n}\}\subset H^1$ such that
\begin{equation}\label{GrindEQ__3_4_}
\inf_{\phi_{\omega}\in \mathcal{G}_{\omega}}\left\|u_{0,n}-\phi_{\omega}\right\|_{H^1}<\frac{1}{n},
\end{equation}
and there exists $\{t_{n}\}\subset \mathbb R$ such that the corresponding solution sequence $\{u_{n}(t_{n})\}$ of \eqref{GrindEQ__1_1_} satisfies
\begin{equation}\label{GrindEQ__3_5_}
\inf_{\phi_{\omega}\in \mathcal{G}_{\omega}}\left\|u_{n}(t_{n})-\phi_{\omega}\right\|_{H^1}\ge \epsilon_{0}.
\end{equation}
By \eqref{GrindEQ__3_4_}, there exists $\{\phi_{n}\}\subset \mathcal{G}_{\omega}$ such that
\begin{equation}\label{GrindEQ__3_6_}
\left\|u_{0,n}-\phi_{n}\right\|_{H^1}<\frac{2}{n}.
\end{equation}
Since $\{\phi_{n}\}\subset \mathcal{G}_{\omega}=\mathcal{M}_{\omega}$, $\{\phi_{n}\}$ is a minimizing sequence of $d_{\omega}$. By repeating the same argument as in the proof of Lemma \ref{lem 2.4.}, there exists $\phi\in \mathcal{G}_{\omega}$ such that, up to a subsequence,
\begin{equation}\label{GrindEQ__3_7_}
\lim_{n\to\infty}\left\|\phi_{n}-\phi\right\|_{H^1}=0.
\end{equation}
\eqref{GrindEQ__3_6_} and \eqref{GrindEQ__3_7_} imply that
\begin{equation}\label{GrindEQ__3_8_}
\lim_{n\to\infty}\left\|u_{0,n}-\phi\right\|_{H^1}=0.
\end{equation}
By \eqref{GrindEQ__3_8_} and Lemma \ref{lem 2.2.}, we can also see that
\begin{equation}\label{GrindEQ__3_9_}
\left\||x|^{\frac{-b}{\sigma+2}}u_{0,n}\right\|_{L^{\sigma+2}}^{\sigma+2}\to\left\||x|^{\frac{-b}{\sigma+2}}\phi\right\|_{L^{\sigma+2}}^{\sigma+2},~\textrm{as}~n\to \infty.
\end{equation}
By \eqref{GrindEQ__1_6_}, \eqref{GrindEQ__3_8_}, \eqref{GrindEQ__3_9_} and the conservations laws, we have
\begin{equation}\label{GrindEQ__3_10_}
M(u_{n}(t_n))\to M(\phi),~E(u_n(t_n))\to E(\phi),~S_{\omega}(u_n(t_n))\to S_{\omega}(\phi),~\textrm{as}~n\to \infty.
\end{equation}
It also follows from \eqref{GrindEQ__2_2_} and \eqref{GrindEQ__3_10_} that
\begin{equation}\label{GrindEQ__3_11_}
H_{\omega}(u_n)\to H(\phi),~ \left\||x|^{\frac{-b}{\sigma+2}}u_{n}\right\|_{L^{\sigma+2}}^{\sigma+2}\to\left\||x|^{\frac{-b}{\sigma+2}}\phi\right\|_{L^{\sigma+2}}^{\sigma+2},
~\textrm{as}~n\to \infty.
\end{equation}
Putting
\begin{equation}\label{GrindEQ__3_12_}
\lambda_{n}:=\left(H_{\omega}(u_n(t_n))\right)^{\frac{1}{\sigma}}\left\||x|^{\frac{-b}{\sigma+2}}u_n(t_n)\right\|_{L^{\sigma+2}}^{-\frac{\sigma+2}{\sigma}},
~\tilde{u}_{n}:=\lambda_{n}u_{n}(t_n),
\end{equation}
we can see that $K_{\omega}(\tilde{u}_{n})=0$. It also follows from \eqref{GrindEQ__2_2_}, \eqref{GrindEQ__3_11_} and the fact that $K_{\omega}(\phi)=0$, we can see that
\begin{equation}\label{GrindEQ__3_13_}
\lim_{n\to \infty}\lambda_{n}=1.
\end{equation}
Using \eqref{GrindEQ__3_11_}--\eqref{GrindEQ__3_13_}, we immediately get
$$
S_{\omega}(\tilde{u}_{n})\to S_{\omega}(\phi)=s(\omega).
$$
Hence $\{\tilde{u}_{n}\}$ is a minimizing sequence of $d_{\omega}$. By using the same argument as in the proof of Lemma \ref{lem 2.4.}, there exits $\tilde{\phi}\in \mathcal{G}_{\omega}$ such that, up to a sequence,
\begin{equation}\label{GrindEQ__3_14_}
\lim_{n\to\infty}\left\|\tilde{u}_{n}-\tilde{\phi}\right\|_{H^1}=0.
\end{equation}
Using \eqref{GrindEQ__3_12_} and \eqref{GrindEQ__3_13_}, we also have
\begin{equation}\label{GrindEQ__3_15_}
\lim_{n\to\infty}\left\|\tilde{u}_{n}-u_{n}(t_n)\right\|_{H^1}=0.
\end{equation}
By \eqref{GrindEQ__3_14_} and \eqref{GrindEQ__3_15_}, we immediately get
\begin{equation}\label{GrindEQ__3_16_}
\lim_{n\to\infty}\left\|u_{n}(t_n)-\tilde{\phi}\right\|_{H^1}=0,
\end{equation}
which contradicts \eqref{GrindEQ__3_5_}. This completes the proof.
\end{proof}

\subsection{Mass-critical case.}
In this subsection, we prove Proposition \ref{prp 1.5.} and Theorem \ref{thm 1.7.}.

\begin{proof}[\textnormal{\textbf{Proof of Proposition \ref{prp 1.5.}}}]
If $T^{*}<\infty$, then we are done. If $T^{*}=\infty$, then we have to show that there exists $t_{n}\to \infty$ such that $\left\|u(t_{n})\right\|_{\dot{H}^{1}}\to \infty$ as $n\to \infty$. Assume by contradiction that it doesn't hold, i.e. $\sup_{t\in [0,\infty)}{\left\|u(t)\right\|_{\dot{H}^{1}}}\le M_{0}$ for some $M_{0}>0$. Using the conservation of mass and the fact $\left\|u\right\|_{\dot{H}^{1}}\sim \left\|u\right\|_{\dot{H}_{c}^{1}}$, we have
$$
\sup_{t\in [0,\infty)}{\left\|u(t)\right\|_{H_{c}^{1}}}\le M_{1},
$$
for some $M_{1}>0$. Hence, it follows from Lemma \ref{lem 3.1.}, the conservation of energy and the fact $\sigma=\frac{4-2b}{d}$ that
$$
V''_{\varphi_{R}}(t)\le 8G(u(t))+CR^{-2}+CR^{-b}\left\|u(t)\right\|_{H_{c}^{1}}^{\sigma+2}\le 16E(u_{0})+CR^{-2}+CR^{-b}M_{1}^{\sigma+2},
$$
for all $t\in[0,\infty)$. By taking $R>1$ large enough, we have for all $t\in[0,\infty)$,
$$
V''_{\varphi_{R}}(t)\le 8E(u_{0})<0.
$$
Integrating this estimate, there exists $t_{0}>0$ sufficiently large such that $V_{\varphi_{R}}(t_{0})<0$ which is impossible. This completes the proof.
\end{proof}

\begin{proof}[\textnormal{\textbf{Proof of Theorem \ref{thm 1.7.}}}]
Let $\omega>0$ and $\phi_{\omega}\in \mathcal{G}_{\omega}$. For $\lambda,\mu>0$, we put
\begin{equation}\label{GrindEQ__3_17_}
u_{\lambda}^{\mu}(x):=\mu\lambda^{\frac{d}{2}}\phi_{\omega}(\lambda x).
\end{equation}
One can easily see that
\begin{eqnarray}\begin{split}\label{GrindEQ__3_18_}
&\left\|u_{\lambda}^{\mu}\right\|_{L^2}=\mu\left\|\phi_{\omega}\right\|_{L^2},~
\left\|\nabla u_{\lambda}^{\mu}\right\|_{L^2}=\mu\left\|\nabla \phi_{\omega}\right\|_{L^2},~
\left\||x|^{-1}u_{\lambda}^{\mu}\right\|_{L^2}=\mu\left\||x|^{-1}\phi_{\omega}\right\|_{L^2},\\
&~~~~~~~~~~~~~~~~~~~\left\||x|^{\frac{-b}{\sigma+2}}u_{\lambda}^{\mu}\right\|_{L^{\sigma+2}}
=\mu\lambda^{\frac{d\sigma+2b}{2(\sigma+2)}}\left\||x|^{\frac{-b}{\sigma+2}}\phi_{\omega}\right\|_{L^{\sigma+2}}.
\end{split}\end{eqnarray}
Using Lemma \ref{lem 3.3.} and \eqref{GrindEQ__3_18_}, we can see that $u_{\lambda}^{\mu}\to \phi_{\omega}$ strongly in $H^1$ as $\lambda\to 1$ and $\mu\to 1$. Hence, there exist $\lambda_{0}>1$ and $\mu_{0}>1$ such that $\left\|u_{\lambda_{0}}^{\mu_{0}}-\phi_{\omega}\right\|_{H^1}<\epsilon$.
Putting $u_{0}(x):=u_{\lambda_{0}}^{\mu_{0}}(x)$ and $\sigma_{*}:=\frac{4-2b}{d}$, it follows from \eqref{GrindEQ__3_3_} and \eqref{GrindEQ__3_18_} that
\begin{eqnarray}\begin{split}\label{GrindEQ__3_19_}
E(u_0)&=\frac{1}{2}\left\|u_0\right\|_{\dot{H}_{c}^1}^2-
        \frac{1}{\sigma_{*}+2}\left\||x|^{\frac{-b}{\sigma_{*}+2}}u_0\right\|_{L^{\sigma_{*}+2}}^{\sigma_{*}+2}\\
&=\frac{1}{2}\mu_{0}^2\lambda_{0}^2\left\|\phi_{\omega}\right\|_{\dot{H}_{c}^1}^2
       -\frac{1}{\sigma_{*}+2}\mu_{0}^{\sigma_{*}+2}\lambda_{0}^{2}\left\||x|^{\frac{-b}{\sigma_{*}+2}}\phi_{\omega}\right\|_{L^{\sigma_{*}+2}}^{\sigma_{*}+2}\\
&=\frac{1}{2}\left(1-\mu_{0}^{\sigma_{*}}\right)\mu_{0}^2\lambda_{0}^2\left\|\phi_{\omega}\right\|_{\dot{H}_{c}^1}^2<0.
\end{split}\end{eqnarray}
By \eqref{GrindEQ__3_19_} and Proposition \ref{prp 1.5.}, the corresponding solution $u$ of \eqref{GrindEQ__1_1_} with initial data $u_{0}$ blows up in finite or infinite time.
\end{proof}

\subsection{Intercritical case.}
In this subsection, we prove Theorem \ref{thm 1.10.}. To this end, we first recall the following criterion for blow-up of solutions to the intercritical INLS$_{c}$ equation \eqref{GrindEQ__1_1_}.
\begin{lemma}[\cite{AJK23}]\label{lem 3.4.}
Let $d\ge3$, $0<b<2$, $\frac{4-2b}{d}<\sigma<\frac{4-2b}{d-2}$ and $c>-c(d)$. Let $u$ be the solution to \eqref{GrindEQ__1_1_} defined on the maximal forward time interval of existence $[0,T^{*})$.
Assume that
  \begin{equation}\label{GrindEQ__3_20_}
   \sup_{t\in [0,T^{*})}{G(u(t))}\le -\delta
  \end{equation}
  for some $\delta>0$, where $G(u)$ is as in \eqref{GrindEQ__3_2_}.
  Then either $T^{*}<\infty$ or $T^{*}=\infty$ and there exists a time sequence $t_{n}\to \infty$ such that $\left\|u(t_{n})\right\|_{\dot{H}_{c}^{1}}\to \infty$ as $n\to \infty$.
\end{lemma}
A similar statement holds for negative time and we omit the details.

\begin{lemma}\label{lem 3.5.}
Let $d\ge 3$, $0<b<2$, $c\neq0$ be such that $c<c(d)$, $\frac{4-2b}{d}<\sigma<\frac{4-2b}{d-2}$ and $\omega>0$. Let $\phi_{\omega}\in \mathcal{G}_{\omega}$. Then we have
$$
S_{\omega}(\phi_{\omega})=\inf\{S_{\omega}(v):~u\in H^1\setminus \{0\},~G(v)=0\},
$$
where $G(\cdot)$ is as in \eqref{GrindEQ__3_2_}.
\end{lemma}
\begin{proof}
Let $s^{*}(\omega):=\inf\{S_{\omega}(v):~v\in H^1\setminus \{0\},~G(v)=0\}$. By Lemma \ref{lem 3.2.}, we can see that $G(\phi_{\omega})=0$. By definition of $s^{*}(\omega)$, we have
\begin{equation}\label{GrindEQ__3_21_}
S_{\omega}(\phi_{\omega})\ge s^{*}(\omega).
\end{equation}
Now, we consider $v\in H^{1}\setminus \{0\}$ be such that $G(v)=0$. If $K_{\omega}(v)=0$, then by Proposition \ref{prp 1.2.}, $S_{\omega}(v)\ge S_{\omega}(\phi_{\omega})$. Let us assume that $K_{\omega}(v)\neq 0$.
Putting $v_{\lambda}(x):=\lambda^{\frac{d}{2}}v(\lambda x)$, it is easy to see that
\begin{eqnarray}\begin{split}\label{GrindEQ__3_22_}
&\left\|v_{\lambda}\right\|_{L^2}=\mu\left\|v\right\|_{L^2},~
\left\|\nabla v_{\lambda}\right\|_{L^2}=\left\|\nabla v\right\|_{L^2},~
\left\||x|^{-1}v_{\lambda}\right\|_{L^2}=\mu\left\||x|^{-1}v\right\|_{L^2},\\
&~~~~~~~~~~~~~~~~~~~\left\||x|^{\frac{-b}{\sigma+2}}v_{\lambda}\right\|_{L^{\sigma+2}}
=\lambda^{\frac{d\sigma+2b}{2(\sigma+2)}}\left\||x|^{\frac{-b}{\sigma+2}}v\right\|_{L^{\sigma+2}}.
\end{split}\end{eqnarray}
Hence, we have
\begin{equation}\label{GrindEQ__3_23_}
S_{\omega}(v_{\lambda})=\frac{\lambda^2}{2}\left\|v\right\|_{\dot{H}_{c}^1}^2+\frac{\omega}{2}\left\|v\right\|_{L^2}^{2}
                       -\frac{\lambda^{\frac{d\sigma+2b}{2}}}{\sigma+2}\left\||x|^{\frac{-b}{\sigma+2}}v\right\|_{L^{\sigma+2}}^{\sigma+2},
\end{equation}
\begin{equation}\label{GrindEQ__3_24_}
G(v)=\partial_{\lambda}S_{\omega}(v_{\lambda})|_{\lambda=1},
\end{equation}
and
\begin{equation}\label{GrindEQ__3_25_}
K_{\omega}(v_{\lambda})=\lambda^2\left\|v\right\|_{\dot{H}_{c}^1}^2+\omega\left\|v\right\|_{L^2}^{2}
                       -\lambda^{\frac{d\sigma+2b}{2}}\left\||x|^{\frac{-b}{\sigma+2}}v\right\|_{L^{\sigma+2}}^{\sigma+2}.
\end{equation}
We can see that
$$\lim_{\lambda\to 0}K_{\omega}(v_{\lambda})=\omega\left\|v\right\|_{L^2}^{2}>0.
$$
Since $\frac{d\sigma+2b}{2}>2$, we also have
$$
\lim_{\lambda\to \infty}K_{\omega}(v_{\lambda})=-\infty.
$$
Hence, there exists $\lambda_{0}>0$ such that $K_{\omega}(v_{\lambda_{0}})=0$.
By Proposition \ref{prp 1.2.}, we see that
\begin{equation}\label{GrindEQ__3_26_}
S_{\omega}(v_{\lambda_{0}})\ge S_{\omega}(\phi_{\omega}).
\end{equation}
Meanwhile, noticing that
\begin{eqnarray}\begin{split}\nonumber
\partial_{\lambda}S_{\omega}(v_{\lambda})
=\lambda\left(\left\|v\right\|_{\dot{H}_{c}^1}^2                    -\frac{d\sigma+2b}{2(\sigma+2)}\lambda^{\frac{d\sigma+2b}{2}-2}\left\||x|^{\frac{-b}{\sigma+2}}v\right\|_{L^{\sigma+2}}^{\sigma+2}\right),
\end{split}\end{eqnarray}
the equation $\partial_{\lambda}S_{\omega}(v_{\lambda})=0$ admits a unique non-zero solution
$$
\lambda_1=\left(\frac{\left\|v\right\|_{\dot{H}_{c}^1}^2}
{\frac{d\sigma+2b}{2(\sigma+2)}\left\||x|^{\frac{-b}{\sigma+2}}v\right\|_{L^{\sigma+2}}^{\sigma+2}}\right)^{\frac{2}{d\sigma+2b-4}},
$$
which is equal to $1$ since $G(v)=0$. We can also see that $\partial_{\lambda}S_{\omega}(v_{\lambda})>0$ if $\lambda\in (0,1)$ and $\partial_{\lambda}S_{\omega}(v_{\lambda})<0$ if $\lambda>1$. In particular, we have $S_{\omega}(v_{\lambda})<S_{\omega}(v_1)=S_{\omega}(v)$ for any $\lambda>0$ and $\lambda\neq 1$. Since $\lambda_{0}>0$, we can see that
\begin{equation}\label{GrindEQ__3_27_}
S_{\omega}(v_{\lambda_{0}})\le S_{\omega}(v).
\end{equation}
In view of \eqref{GrindEQ__3_26_} and \eqref{GrindEQ__3_27_}, we can see that $S_{\omega}(v)\ge S_{\omega}(\phi_{\omega})$ for any $v\in H^{1}\setminus \{0\}$ with $G(v)=0$. Taking the infimum, we have
\begin{equation}\label{GrindEQ__3_28_}
S_{\omega}(\phi_{\omega})\le s^{*}(\omega).
\end{equation}
Using \eqref{GrindEQ__3_21_} and \eqref{GrindEQ__3_28_}, we get the desired result.
\end{proof}

\begin{lemma}\label{lem 3.6.}
Let $d\ge 3$, $0<b<2$, $\frac{4-2b}{d}<\sigma<\frac{4-2b}{d-2}$, $\omega>0$ and $c\neq 0$ be such that $c<c(d)$. Let $\phi_{\omega}\in \mathcal{G}_{\omega}$. Then
$$
\mathcal{B}_{\omega}:=\{v\in H^1\setminus \{0\}:~S_{\omega}(v)<S_{\omega}(\phi_{\omega}),~G(v)<0\}\}
$$
is invariant under the flow of \eqref{GrindEQ__1_1_}, that is, if $u_{0}\in \mathcal{B}_{\omega}$, then the corresponding solution $u(t)$ to \eqref{GrindEQ__1_1_} with $u(0)=u_{0}$, defined on the maximal interval $(-T_{\min},T_{\max})$, satisfies $u(t)\in \mathcal{B}_{\omega}$ for any $t\in (-T_{\min},T_{\max})$.
\end{lemma}
\begin{proof}
Let $u_{0}\in \mathcal{B}_{\omega}$. Using the conservation of mass and energy, we have
\begin{equation}\label{GrindEQ__3_29_}
S_{\omega}(u(t))=S_{\omega}(u_0)< S_{\omega}(\phi_{\omega}),
\end{equation}
for $t\in (-T_{\min},T_{\max})$. It suffices to prove that $G(u(t))<0$ for any $t\in (-T_{\min},T_{\max})$.
Without loss of generality, we assume that there exists $t_{0}\in [0,T_{\max})$ such that $G(u(t_0))\ge 0$. By the continuity of $t\mapsto G(u(t))$, there exists $t_{1}\in (0,T_{0}]$ such that $G(u(t_{1}))=0$. By Lemma \ref{lem 3.5.}, $S_{\omega}(u(t_1))\ge S_{\omega}(\phi_{\omega})$ which contradicts to \eqref{GrindEQ__3_29_}. This completes the proof.
\end{proof}

\begin{lemma}\label{lem 3.7.}
Let $d\ge 3$, $0<b<2$, $c\neq$ be such that $c<c(d)$, $\frac{4-2b}{d}<\sigma<\frac{4-2b}{d-2}$ and $\omega>0$. Let $\phi_{\omega}\in \mathcal{G}_{\omega}$. If $v\in \mathcal{B}_{\omega}$, then we have
$$
G(v)\le 2(S_{\omega}(v)-S_{\omega}(\phi_{\omega})).
$$
\end{lemma}
\begin{proof}
As in the proof of Lemma \ref{lem 3.5.}, putting $v_{\lambda}(x):=\lambda^{\frac{d}{2}}v(\lambda x)$, it follows from \eqref{GrindEQ__3_23_} that
\begin{equation}\label{GrindEQ__3_30_}
\partial_{\lambda}S_{\omega}(v_{\lambda})=\lambda\left\|v\right\|_{\dot{H}_{c}^1}^2
                       -\frac{d\sigma+2b}{2(\sigma+2)}\lambda^{\frac{d\sigma+2b}{2}-1}\left\||x|^{\frac{-b}{\sigma+2}}v\right\|_{L^{\sigma+2}}^{\sigma+2}
=\frac{G(v_{\lambda})}{\lambda}
\end{equation}
and
\begin{equation}\nonumber
\partial_{\lambda}(\lambda \partial_{\lambda}S_{\omega}(v_{\lambda}))
=2\partial_{\lambda}S_{\omega}(v_{\lambda})
-\frac{(d\sigma+2b)(d\sigma+2b-4)}{4(\sigma+2)}\lambda^{\frac{d\sigma+2b}{2}-1}\left\||x|^{\frac{-b}{\sigma+2}}v\right\|_{L^{\sigma+2}}^{\sigma+2}.
\end{equation}
Since $\sigma>\frac{4-2b}{d}$, we have
\begin{equation}\label{GrindEQ__3_31_}
\partial_{\lambda}(\lambda \partial_{\lambda}S_{\omega}(v_{\lambda}))\le 2\partial_{\lambda}S_{\omega}(v_{\lambda}),~\forall \lambda>0.
\end{equation}
By \eqref{GrindEQ__3_30_} and the fact that $G(v)<0$, there exists $\lambda_{0}\in (0,1)$ such that
$$
G(v_{\lambda_0})=\lambda_{0}\partial_{\lambda}S_{\omega}(v_{\lambda})|_{\lambda=\lambda_{0}}=0.
$$
Hence, integrating \eqref{GrindEQ__3_30_} on $[\lambda_{0},1]$, we get
$$
G(v)\le 2(S_{\omega}(v)-S_{\omega}(v_{\lambda_0}))\le 2(S_{\omega}(v)-S_{\omega}(\phi_{\omega})).
$$
Here the last inequality follows from Lemma \ref{lem 3.5.}. This completes the proof.
\end{proof}

We are ready to prove Theorem \ref{thm 1.10.}
\begin{proof}[\textnormal{\textbf{Proof of Theorem \ref{thm 1.10.}}}]
Let $\omega>0$, $\phi_{\omega}\in \mathcal{G}_{\omega}$ and $\lambda>0$. As in the proof of Lemmas \ref{lem 3.5.} and \ref{lem 3.7.}, putting $\phi_{\lambda}(x):=\lambda^{\frac{d}{2}}\phi_{\omega}(\lambda x)$, it is easy to see that
\begin{eqnarray}\begin{split}\label{GrindEQ__3_32_}
&\left\|\phi_{\lambda}\right\|_{L^2}=\left\|\phi_{\omega}\right\|_{L^2},~
\left\|\nabla \phi_{\lambda}\right\|_{L^2}=\left\|\nabla \phi_{\omega}\right\|_{L^2},~
\left\||x|^{-1}\phi_{\lambda}\right\|_{L^2}=\left\||x|^{-1}\phi_{\omega}\right\|_{L^2},\\
&~~~~~~~~~~~~~~~~~~~\left\||x|^{\frac{-b}{\sigma+2}}\phi_{\lambda}\right\|_{L^{\sigma+2}}
=\lambda^{\frac{d\sigma+2b}{2(\sigma+2)}}\left\||x|^{\frac{-b}{\sigma+2}}\phi_{\omega}\right\|_{L^{\sigma+2}}.
\end{split}\end{eqnarray}
Using Lemma \ref{lem 3.3.} and \eqref{GrindEQ__3_32_}, we can see that $\phi_{\lambda}\to \phi_{\omega}$ strongly in $H^1$ as $\lambda\to 1$. Hence, there exists $\lambda_{0}>1$ such that $\left\|\phi_{\lambda_{0}}-\phi_{\omega}\right\|_{H^1}<\epsilon$.
First, we claim that $\phi_{\lambda_{0}}\in \mathcal{B}_{\omega}$. Indeed, it follows from Lemma \ref{lem 3.2.} that $G(\phi_{\omega})=0$.
Using the same argument as in the proof of Lemma \ref{lem 3.5.}, we can see that $\partial_{\lambda}S_{\omega}(\phi_{\lambda})>0$ if $\lambda\in (0,1)$ and $\partial_{\lambda}S_{\omega}(\phi_{\lambda})<0$ if $\lambda>1$. Hence, we have $S_{\omega}(\phi_{\lambda})<S_{\omega}(\phi_{\omega})$ for any $\lambda>0$, $\lambda\neq1$. Noticing $\lambda_{0}>1$ and $G(\phi_{\lambda})=\lambda\partial_{\lambda}S_{\omega}(\phi_{\lambda})$, we have
\begin{eqnarray}\begin{split}\label{GrindEQ__3_33_}
S_{\omega}(\phi_{\lambda_{0}})<S_{\omega}(\phi_{\omega}),~G(\phi_{\lambda_{0}})<0,
\end{split}\end{eqnarray}
from which we get $\phi_{\lambda_{0}}\in \mathcal{B}_{\omega}$.
From the local theory of \cite{S16}, there exists a unique maximal solution $u(t)\in C((-T_{\min},T_{\max}),H^1)$ to \eqref{GrindEQ__1_1_} with initial data $u(0)=u_{0}:=\phi_{\lambda_{0}}$.
Let $u_{0}\in \mathcal{B}_{\omega}$. Using \eqref{GrindEQ__3_33_}, the conservation of mass and energy, we have
\begin{equation}\label{GrindEQ__3_34_}
S_{\omega}(u(t))=S_{\omega}(\phi_{\lambda_{0}}),
\end{equation}
for $t\in (-T_{\min},T_{\max})$.
By Lemma \ref{lem 3.6.}, we can see that $u(t)\in \mathcal{B}_{\omega}$ for any $t\in (-T_{\min},T_{\max})$.
Hence, using Lemma \ref{lem 3.7.} and \eqref{GrindEQ__3_34_}, we have
$$
G(u(t))\le 2(S_{\omega}(u(t))-S_{\omega}(\phi_{\omega}))=(S_{\omega}(\phi_{\lambda_{0}})-S_{\omega}(\phi_{\omega}))<0,
$$
for any $t\in (-T_{\min},T_{\max})$. By Lemma \ref{lem 3.4.}, $u$ blows up in finite or infinite time. This completes the proof.
\end{proof}


\smallskip

\emph{E-mail address}: jm.an0221@ryongnamsan.edu.kp; cioc12@ryongnamsan.edu.kp (J. An)

\emph{E-mail address}: math9@ryongnamsan.edu.kp (H. Mun)

\emph{E-mail address}: jm.kim0211@ryongnamsan.edu.kp (J. Kim)

\end{document}